\documentclass[10pt,twoside]{siamltex}
\usepackage{amsfonts,epsfig}
 \usepackage{graphics}
 \usepackage{amssymb}
\newtheorem{example}[theorem]{Example}

\begin{document}



\bibliographystyle{plain}
\title{A new improved error bound for linear complementarity problems for $B$-matrices\thanks{Received by the editors on
Month x, 201x. Accepted for publication on Month y, 201y   Handling
Editor: .}}

\author{
Lei Gao\thanks{School of Mathematics and Information Science, Baoji
University of Arts and Sciences, Baoji, Shannxi, 721007, P. R. China
(gaolei@bjwlxy.edu.cn).}
\and Chaoqian Li\thanks{School of Mathematics and Statistics, Yunnan
University, Kunming, Yunnan, 650091, P.R. China
(lichaoqian@ynu.edu.cn).} }

\pagestyle{myheadings} \markboth{L. GAO, C.Q. LI}{A new improved
error bound for linear complementarity problems for $B$-matrices}
\maketitle

\begin{abstract}
A new error bound for the linear complementarity problem when the
matrix involved is a $B$-matrix is presented, which improves the
corresponding result  in [C.Q. Li et al., A new error bound for
linear complementarity problems for $B$-matrices. Electron. J.
Linear Al., 31:476-484, 2016]. In addition some sufficient
conditions such that the new bound is sharper than that in [M.
Garc\'{i}a-Esnaola and J.M. Pe\~{n}a. Error bounds for linear
complementarity problems for $B$-matrices. Appl. Math. Lett.,
22:1071-1075, 2009] are provided.
\end{abstract}


\section{Introduction} \label{intro-sec}
Given an $n\times n$ real matrix $M$ and $q\in R^n$, the linear
complementarity problem (LCP) is to find a vector $x\in R^{n}$
satisfying
\begin{equation}\label{eq1}
x\geqslant 0, Mx+q\geqslant 0, (Mx+q)^Tx=0 \end{equation} or to show
that no such vector $x$ exists. We denote this problem (\ref{eq1})
by LCP$(M, q)$. The LCP$(M, q)$ arises in many applications such as
finding Nash equilibrium point of a bimatrix game, the network
equilibrium problems, the contact problems and the free boundary
problems for journal bearing etc, for details, see \cite{Ber,Co,Mu}.

It is well-known that the LCP$(M, q)$ has a unique solution for any
vector $q\in R^n$ if and only if $M$ is a $P$-matrix \cite{Co}. Here
a matrix $M$ is called a $P$-matrix if all its principal minors are
positive. For the LCP$(M, q)$, one of the interesting problems is to
estimate
\begin{equation}\label{e_b}\max\limits_{d\in [0,1]^n}||(I -D+DM)^{-1}||_{\infty},\end{equation} which
can be used to bound the error $||x-x^*||_{\infty}$ \cite{Chen},
that is,
\[ ||x-x^*||_{\infty} \leqslant \max\limits_{d\in [0,1]^n}||(I -D+DM)^{-1}||_{\infty} ||r(x)||_\infty,\]
where $x^*$ is the solution of the LCP$(M, q)$, $r(x)=\min\{
x,Mx+q\}$, $D=diag(d_i)$ with $0\leqslant d_i \leqslant 1$ for each
$i\in N$, $d=[d_1,d_2,...,d_n]^{T}\in [0,1]^n$, and the min operator
$r(x)$ denotes the componentwise minimum of two vectors.

When the matrix $M$ for the LCP$(M, q)$ belongs to $P$-matrices or
some subclass of $P$-matrices, various bounds for (\ref{e_b}) were
proposed, e.g., see
\cite{Che,Chen0,Chen,Dai,Dai1,Dai2,Ga,Ga0,Ga1,Ga2,Li} and references
therein. Recently, Garc\'{i}a-Esnaola and  Pe\~na  in \cite{Ga}
provided an upper bound for (\ref{e_b}) when $M$ is a $B$-matrix as
a subclass of $P$-matrices. Here, a matrix $M=[m_{ij}]\in R^{n, n}$
is called a $B$-matrix \cite{Jp} if for each $i\in
N=\{1,2,\ldots,n\}$,
\begin{eqnarray*}
\sum\limits_{k\in N}m_{ik}>0, ~and
~\frac{1}{n}\left(\sum\limits_{k\in N}m_{ik}\right)>m_{ij} ~~for~
any~ j\in N ~and~ j\neq i.\end{eqnarray*}

\begin{theorem}\emph{\cite[Theorem 2.2]{Ga}} \label{G-B}  Let $M=[m_{ij}]\in R^{n, n}$
be a $B$-matrix with the form
\begin{eqnarray}\label{eq3}
M=B^{+}+C,
\end{eqnarray}
where
\begin{eqnarray}\label{eq4} B^+ =[b_{ij}]= \left[ \begin{array}{ccc}
 m_{11}-r_1^+    &\cdots     &m_{1n}-r_1^+ \\
 \vdots          &          &\vdots  \\
 m_{n1}-r_n^+     &\cdots    &m_{nn}-r_n^+
\end{array} \right], ~C=\left[ \begin{array}{ccc}
 r_1^+    &\cdots     &r_1^+ \\
 \vdots          &          &\vdots  \\
r_n^+     &\cdots    &r_n^+
\end{array} \right], \end{eqnarray}
and $r_i^+=\max\{ 0,m_{ij}|j\neq i\}$. Then
\begin{equation} \label{G-bound}\max\limits_{d\in [0,1]^n}||(I-D+DM)^{-1}||_{\infty}
\leqslant \frac{n-1}{\min\{\beta,1\}},\end{equation} where
$\beta=\min\limits_{i\in N}\{\beta_i\}$ and
$\beta_i=b_{ii}-\sum\limits_{j\neq i}|b_{ij}|$.
\end{theorem}

It is not difficult to see that the bound (\ref{G-bound}) will be
inaccurate when the matrix $M$ has very small value of
$\min\limits_{i\in N}\{b_{ii}-\sum\limits_{j\neq i}|b_{ij}|\}$, for
details, see \cite{Li1,Li2}. To conquer this problem, Li et al., in
\cite{Li3} gave the following bound for (\ref{e_b}) when $M$ is a
$B$-matrix, which improves those provided by Li and Li in
\cite{Li1,Li2}.

\begin{theorem}\emph{\cite[Theorem 2.4]{Li3}} \label{L-C}Let $M=[m_{ij}]\in R^{n, n}$
be a $B$-matrix with the form $M=B^{+}+C$, where $B^{+}=[b_{ij}]$ is
the matrix of (4). Then
\begin{eqnarray}\label{lc-bound} \max\limits_{d\in [0,1]^n}||(I-D+DM)^{-1}||_{\infty} \leqslant
\sum\limits_{i=1}^{n}\frac{n-1}{\min\{\bar{\beta}_i,1\}}\prod\limits_{j=1}^{i-1}\frac{b_{jj}}{\bar{\beta}_j},
\end{eqnarray}
where
$\bar{\beta}_i=b_{ii}-\sum\limits_{j=i+1}^{n}|b_{ij}|l_i(B^{+})$
with $l_k(B^{+})=\max\limits_{k\leq i\leq
n}\left\{\frac{1}{|b_{ii}|}\sum\limits_{j=k,\atop j\neq
i}^{n}|b_{ij}|\right\}$, and
$\prod\limits_{j=1}^{i-1}\frac{b_{jj}}{\bar{\beta}_j}=1$ if $i=1$.
\end{theorem}

In this paper, we further improve error bounds on the LCP$(M,q)$
when $M$ belongs to $B$-matrices. The rest of this paper is
organized as follows: In Section 2 we present a new error bound for
(\ref{e_b}), and then prove that this bound are better than those in
Theorems 1 and 2. In Section 3, some numerical examples are given to
illustrate our theoretical results obtained.

\section{A new error bound for the  LCP$(M, q)$ of $B$-matrices}
In this section, an upper bound for (\ref{e_b}) is provided when $M$
is a $B$-matrix. Firstly, some definitions, notation and lemmas
which will be used later are given as follows.

A matrix $A=[a_{ij}]\in C^{n,n}$ is called a strictly diagonally
dominant ($SDD$) matrix if $|a_{ii}|>\sum\limits_{j\neq
i}^{n}|b_{ij}|$ for all $i=1,2,\ldots,n$. A matrix $A=[a_{ij}]$ is
called a nonsingular $M$-matrix if its inverse is nonnegative and
all its off-diagonal entries are nonpositive \cite{Ber}. In
\cite{Jp} it was proved that a $B$-matrix has positive diagonal
elements, and a real matrix $A$ is a $B$-matrix if and only if it
can be written in form (\ref{eq3}) with $B^{+}$ being a $SDD$
matrix. Given a matrix $A=[a_{ij}]\in C^{n,n}$, let
\begin{eqnarray}\label{eq7}w_{ij}(A)&=&\frac{|a_{ij}|}{|a_{ii}|-\sum\limits_{k=j+1,\atop k\neq i}^{n}|a_{ik}|},~
i\neq j,\nonumber\\
w_i(A)&=&\max\limits_{j\neq i}\{w_{ij}\},\\
m_{ij}(A)&=&\frac{|a_{ij}|+\sum\limits_{ k=j+1,\atop k\neq
i}^{n}|a_{ik}|w_k}{|a_{ii}|}, ~i\neq j.\nonumber
\end{eqnarray}

\begin{lemma}\emph{\cite[Theorem 14]{Yang}} \label{lem1} Let $A=[a_{ij}]$
be an $n\times n$ row strictly diagonally dominant $M$-matrix. Then
\[||A^{-1}||_{\infty}\leqslant \sum\limits_{i=1}^{n}\left
(\frac{1}{a_{ii}-\sum\limits_{k=i+1}^{n}|a_{ik}|m_{ki}(A)}\prod\limits_{j=1}^{i-1}\frac{1}{1-u_j(A)l_j(A)}\right),
\]
where $u_i(A)=\frac{1}{|a_{ii}|}\sum\limits_{j=i+1}^{n}|a_{ij}|$,
$l_k(A)=\max\limits_{k\leq i\leq
n}\left\{\frac{1}{|a_{ii}|}\sum\limits_{j=k,\atop j\neq
i}^{n}|a_{ij}|\right\}$,
$\prod\limits_{j=1}^{i-1}\frac{1}{1-u_j(A)l_j(A)}=1$ if $i=1$, and
$m_{ki}(A)$ is defined as in (\ref{eq7}).
\end{lemma}

\begin{lemma} \cite[Lemma 3]{Li1} \label{lem2}
Let $\gamma > 0$ and $ \eta \geqslant 0 $. Then for any $x\in
[0,1]$, \[ \frac{1}{1-x+\gamma x} \leqslant
\frac{1}{\min\{\gamma,1\}}\] and \[  \frac{\eta x}{1-x+\gamma x}
\leqslant  \frac{\eta }{\gamma}.\]
\end{lemma}

\begin{lemma} \cite[Lemma 5]{Li2} \label{lem3}
Let $A=[a_{ij}]$ with $a_{ii}>\sum\limits_{j=i+1}^{n}|a_{ij}|$ for
each $i\in N$. Then for any $x_i\in [0,1]$,
\[\frac{1-x_i+a_{ii}x_i}{1-x_i+a_{ii}x_i-\sum\limits_{j=i+1}^{n}|a_{ij}|x_i} \leqslant
 \frac{a_{ii}}{a_{ii}-\sum\limits_{j=i+1}^{n}|a_{ij}|}.\]
\end{lemma}
Lemmas \ref{lem2} and \ref{lem3} will be used in the proofs of the
following lemma and of Theorem \ref{GL-mian}.

\begin{lemma} \label{lem4}Let $M=[m_{ij}]\in R^{n, n}$
be a $B$-matrix with the form $M=B^{+}+C$, where $B^{+}=[b_{ij}]$ is
the matrix of (\ref{eq4}). And let
$B_{D}^{+}=I-D+DB^{+}=[\tilde{b}_{ij}]$ where $D=diag(d_i)$ with
$0\leqslant d_i \leqslant 1$. Then \[ w_i(B_{D}^{+})\leqslant
\max\limits_{j\neq
i}\left\{\frac{|b_{ij}|}{b_{ii}-\sum\limits_{k=j+1,\atop k\neq
i}^{n}|b_{ik}|}\right\}\] and \[ m_{ij}(B_{D}^{+})\leqslant
\tilde{m}_{ij}(B^{+}),\] where $w_i(B_{D}^{+})$, $m_{ij}(B_{D}^{+})$
are defined as in (\ref{eq7}), and
\[
\tilde{m}_{ij}(B^{+})=\frac{1}{b_{ii}}\left(|b_{ij}|+\sum\limits_{k=j+1,\atop
k\neq i}^{n}\left(|b_{ik}|\cdot\max\limits_{h\neq
k}\left\{\frac{|b_{kh}|}{b_{kk}-\sum\limits_{l=h+1,\atop l\neq
k}^{n}|b_{kl}|}\right\}\right)\right).\]
\end{lemma}

\begin{proof}
Note that
\[[B_{D}^{+}]_{ij}=\tilde{b}_{ij}=\left\{ \begin{array}{cc}
  1-d_i+d_ib_{ij},  &i=j,\\
 d_ib_{ij},  &i\neq j.
\end{array} \right.\]
Since $B^{+}$ is $SDD$, $b_{ii}-\sum\limits_{k=j+1,\atop k\neq
i}^{n}|b_{ik}|> 0$. Hence, by Lemma \ref{lem2} and (\ref{eq7}), it
follows that
\begin{eqnarray} \label{eq8} w_i(B_{D}^{+})=\max\limits_{j\neq
i}\left\{w_{ij}(B_{D}^{+})\right\} &=&\max\limits_{j\neq
i}\left\{\frac{|b_{ij}|d_i}{1-d_i+b_{ii}d_i-\sum\limits_{k=j+1,\atop
k\neq
i}^{n}|b_{ik}|d_i}\right\}\nonumber\\
&\leqslant&\max\limits_{j\neq
i}\left\{\frac{|b_{ij}|}{b_{ii}-\sum\limits_{k=j+1,\atop k\neq
i}^{n}|b_{ik}|}\right\}.\end{eqnarray} Furthermore, it follows from
(\ref{eq7}), (\ref{eq8}) and Lemma \ref{lem2} that for each $i\neq
j~(j<i\leqslant n)$
\begin{eqnarray*}  m_{ij}(B_{D}^{+})&=&\frac{|b_{ij}|\cdot
d_i+\sum\limits_{k=j+1,\atop k\neq i}^{n}|b_{ik}|\cdot d_i\cdot
w_k(B_{D}^{+})}{1-d_i+b_{ii}\cdot
d_i}\nonumber\\
&\leqslant&
\frac{1}{b_{ii}}\cdot\left(|b_{ij}|+\sum\limits_{k=j+1,\atop k\neq
i}^{n}|b_{ik}|\cdot w_k(B_{D}^{+})\right)\nonumber \\
&\leqslant&\frac{1}{b_{ii}}\left(|b_{ij}|+\sum\limits_{k=j+1,\atop
k\neq i}^{n}\left(|b_{ik}|\cdot\max\limits_{h\neq
k}\left\{\frac{|b_{kh}|}{b_{kk}-\sum\limits_{l=h+1,\atop l\neq
k}^{n}|b_{kl}|}\right\}\right)\right)\nonumber\\
&=&\tilde{m}_{ij}(B^{+}).\end{eqnarray*} The proof is completed.
\end{proof}

By Lemmas \ref{lem1}, \ref{lem2}, \ref{lem3} and \ref{lem4}, we give
the following bound for (\ref{e_b}) when $M$ is a $B$-matrix.

\begin{theorem} \label{GL-mian} Let $M=[m_{ij}]\in R^{n, n}$ be a $B$-matrix with the form $M=B^{+}+C$, where $B^{+}=[b_{ij}]$ is the matrix of
(\ref{eq4}). Then
\begin{eqnarray} \label{Gl-bound}\max\limits_{d\in [0,1]^n}||(I-D+DM)^{-1}||_{\infty} \leqslant
\sum\limits_{i=1}^{n}\frac{n-1}{\min\{\widehat{\beta}_i,1\}}\prod\limits_{j=1}^{i-1}\frac{b_{jj}}{\bar{\beta}_j},
\end{eqnarray}
where
$\widehat{\beta}_i=b_{ii}-\sum\limits_{k=i+1}^{n}|b_{ik}|\cdot\tilde{m}_{ki}(B^{+})$
with $\tilde{m}_{ki}(B^{+})$ is defined in lemma \ref{lem4},
$\bar{\beta}_i$ is defined in Theorem \ref{L-C}, and
$\prod\limits_{j=1}^{i-1}\frac{b_{jj}}{\bar{\beta}_j}=1$ if $i=1$.
\end{theorem}

\begin{proof} Let $M_D=I-D+DM$. Then
\[M_D=I-D+DM=I-D+D(B^{+}+C)=B_D^{+}+C_D,\] where $B_D^{+}=I-D+DB^{+}=[\tilde{b}_{ij}]$ and $C_D=DC$.
Similarly to the proof of Theorem 2.2 in \cite{Ga}, we can obtain
that $B_D^{+}$ is an $SDD$ $M$-matrix with positive diagonal
elements and that
\begin{equation}\label{eq10}
||M_D^{-1}||_\infty \leqslant ||\big(I + (B^+_
D)^{-1}\tilde{C}_D\big)^{-1} ||_\infty ||(B^+_D )^{-1} ||_\infty
\leqslant (n-1)  ||(B^+_D )^{-1} ||_\infty.
\end{equation}

Next, we give an upper bound for $ ||(B^+_D )^{-1} ||_\infty$. By
Lemma \ref{lem1}, we have \[ ||(B^+_D )^{-1} ||_\infty \leqslant
\sum\limits_{i=1}^{n}\left
(\frac{1}{1-d_i+b_{ii}d_i-\sum\limits_{k=i+1}^{n}|b_{ik}|\cdot
d_i\cdot
m_{ki}(B^+_D)}\prod\limits_{j=1}^{i-1}\frac{1}{1-u_j(B^+_D)l_j(B^+_D)}\right).\]
where
\[
u_j(B^+_D)=\frac{\sum\limits_{k=j+1}^{n}|b_{jk}|d_j}{1-d_j+b_{jj}d_j},~
l_k(B^+_D)=\max\limits_{k\leq i\leq
n}\left\{\frac{\sum\limits_{j=k,\atop j\neq
i}^{n}|b_{ij}|d_i}{1-d_i+b_{ii}d_i}\right\}, \] and
\begin{eqnarray*}
m_{ki}(B^+_D)=\frac{|b_{ki}|\cdot d_k+\sum\limits_{l=i+1,\atop l\neq
k}^{n}|b_{kl}|\cdot d_k\cdot w_l(B_{D}^{+})}{1-d_k+b_{kk}\cdot d_k}
\end{eqnarray*}
with $w_l(B_{D}^{+})=\max\limits_{h\neq
l}\left\{\frac{|b_{lh}|d_l}{1-d_l+b_{ll}d_l-\sum\limits_{s=h+1,\atop
s\neq l}^{n}|b_{ls}|d_l}\right\}$.

By Lemmas \ref{lem2} and \ref{lem4}, we can easily get that for each
$i\in N$,
\begin{eqnarray}\label{eq11}
\frac{1}{1-d_i+b_{ii}d_i-\sum\limits_{k=i+1}^{n}|b_{ik}|\cdot
d_i\cdot m_{ki}(B^+_D)}&\leqslant&
\frac{1}{\min\left\{b_{ii}-\sum\limits_{k=i+1}^{n}|b_{ik}|\cdot
m_{ki}(B^+_D),1\right\}}\nonumber\\&\leqslant&
\frac{1}{\min\left\{b_{ii}-\sum\limits_{k=i+1}^{n}|b_{ik}|\cdot
\tilde{m}_{ki}(B^+),1\right\}}\nonumber\\&=&\frac{1}{\min\left\{\widehat{\beta}_i,1\right\}},
\end{eqnarray}
and that for each $k\in N$,
\begin{eqnarray}\label{eq12}
l_k(B^+_D)=\max\limits_{k\leq i\leq
n}\left\{\frac{\sum\limits_{j=k,\atop j\neq
i}^{n}|b_{ij}|d_i}{1-d_i+b_{ii}d_i}\right\}\leqslant
\max\limits_{k\leq i\leq
n}\left\{\frac{1}{b_{ii}}\sum\limits_{j=k,\atop j\neq
i}^{n}|b_{ij}|\right\}=l_k(B^+)<1.
\end{eqnarray}
Furthermore, according to Lemma \ref{lem3} and (\ref{eq12}), it
follows that for each $j\in N$,
\begin{eqnarray}\label{eq13}
\frac{1}{1-u_j(B^+_D)l_j(B^+_D)}=\frac{1-d_j+b_{jj}d_j}{1-d_j+b_{jj}d_j-\sum\limits_{k=j+1}^{n}|b_{jk}|\cdot
d_j\cdot l_j(B^+_D)}\leqslant\frac{b_{jj}}{\bar{\beta}_j}.
\end{eqnarray}
By (\ref{eq11}) and (\ref{eq13}), we have
\begin{eqnarray}\label{eq14}
||(B^+_D )^{-1} ||_\infty \leqslant
\frac{1}{\min\left\{\widehat{\beta}_1,1\right\}}+\sum\limits_{i=2}^{n}\left
(\frac{1}{\min\left\{\widehat{\beta}_i,1\right\}}\prod\limits_{j=1}^{i-1}\frac{b_{jj}}{\bar{\beta}_j}\right).
\end{eqnarray}
The conclusion follows from (\ref{eq10}) and (\ref{eq14}).
\end{proof}

The comparisons of the bounds in Theorems \ref{L-C} and
\ref{GL-mian} are established as follows.
\begin{theorem} \label{LG-mian2} Let $M=[m_{ij}]\in R^{n, n}$ be
a $B$-matrix with the form $M=B^{+}+C$, where $B^{+}=[b_{ij}]$ is
the matrix of (\ref{eq4}). Let $\bar{\beta}_i$ and
$\widehat{\beta}_i$ be defined in Theorems \ref{L-C} and
\ref{GL-mian}, respectively. Then
\begin{eqnarray*} \sum\limits_{i=1}^{n}\frac{n-1}{\min\{\widehat{\beta}_i,1\}}\prod\limits_{j=1}^{i-1}\frac{b_{jj}}{\bar{\beta}_j} \leqslant
\sum\limits_{i=1}^{n}\frac{n-1}{\min\{\bar{\beta}_i,1\}}\prod\limits_{j=1}^{i-1}\frac{b_{jj}}{\bar{\beta}_j}.
\end{eqnarray*}
\end{theorem}

\begin{proof}
Note that
\[\bar{\beta}_i=b_{ii}-\sum\limits_{j=i+1}^{n}|b_{ij}|l_i(B^{+}), \widehat{\beta}_i=b_{ii}-\sum\limits_{k=i+1}^{n}|b_{ik}|\tilde{m}_{ki}(B^{+}),\]
and $B^{+}$ is a $SDD$ matrix, it follows that for each $i\neq
j~(j<i\leqslant n)$
\begin{eqnarray*} \tilde{m}_{ij}(B^{+})&=&\frac{1}{b_{ii}}\left(|b_{ij}|+\sum\limits_{k=j+1,\atop k\neq
i}^{n}\left(|b_{ik}|\cdot\max\limits_{h\neq
k}\left\{\frac{|b_{kh}|}{b_{kk}-\sum\limits_{l=h+1,\atop l\neq
k}^{n}|b_{kl}|}\right\}\right)\right)\\&<&
\frac{1}{b_{ii}}\sum\limits_{k=j,\atop k\neq
i}^{n}|b_{ik}|\\&\leqslant& \max\limits_{j\leqslant i\leqslant
n}\left\{\frac{1}{b_{ii}}\sum\limits_{k=j,\atop k\neq
i}^{n}|b_{ik}|\right\}=l_j(B^{+}).
\end{eqnarray*}
Hence, for each $i\in N$
 \[\widehat{\beta}_i=b_{ii}-\sum\limits_{k=i+1}^{n}|b_{ik}|\tilde{m}_{ki}(B^{+})>b_{ii}-\sum\limits_{k=i+1}^{n}|b_{ik}|l_i(B^{+})=\bar{\beta}_i,\]
 which implies that
\[\frac{1}{\min\{\widehat{\beta}_i,1\}}\leqslant
\frac{1}{\min\{\bar{\beta}_i,1\}}.\] This completes the proof.
\end{proof}

Remark here that when $\bar{\beta}_i<1$ for all $i\in N$, then
\begin{eqnarray*}\frac{1}{\min\{\widehat{\beta}_i,1\}}<
\frac{1}{\min\{\bar{\beta}_i,1\}},\end{eqnarray*} which yields that
\begin{eqnarray*} \sum\limits_{i=1}^{n}\frac{n-1}{\min\{\widehat{\beta}_i,1\}}\prod\limits_{j=1}^{i-1}\frac{b_{jj}}{\bar{\beta}_j}
<\sum\limits_{i=1}^{n}\frac{n-1}{\min\{\bar{\beta}_i,1\}}\prod\limits_{j=1}^{i-1}\frac{b_{jj}}{\bar{\beta}_j}.
\end{eqnarray*}

Next it is proved that the bound (\ref{Gl-bound}) given in Theorem
\ref{GL-mian} can improve the bound (\ref{G-bound}) in Theorem
\ref{G-B} (Theorem 2.2 in \cite{Ga}) in some cases.

\begin{theorem}\label{th5}
Let $M=[m_{ij}]\in R^{n, n}$ be a $B$-matrix with the form
$M=B^{+}+C$, where $B^{+}=[b_{ij}]$ is the matrix of (\ref{eq4}).
Let $\beta$, $\bar{\beta}_i$ and $\widehat{\beta}_i$ be defined in
Theorems \ref{G-B}, \ref{L-C} and \ref{GL-mian}, respectively, and
let
$\alpha=1+\sum\limits_{i=2}^{n}\prod\limits_{j=1}^{i-1}\frac{b_{jj}}{\bar{\beta}_j}$
and $\widehat{\beta}=\min\limits_{i\in N}\{\widehat{\beta}_i\}$. If
one of the following conditions holds:

(i) $\widehat{\beta}>1$ and $\alpha<\frac{1}{\beta}$;

(ii) $\widehat{\beta}<1$ and $\alpha\beta<\widehat{\beta}$,\\ then
\begin{eqnarray*}
\sum\limits_{i=1}^{n}\frac{n-1}{\min\{\widehat{\beta}_i,1\}}\prod\limits_{j=1}^{i-1}\frac{b_{jj}}{\bar{\beta}_j}<\frac{n-1}{\min\{\beta,1\}}.
\end{eqnarray*}
\end{theorem}

\begin{proof}
When $\widehat{\beta}>1$ and $\alpha<\frac{1}{\beta}$, we can easily
get
\begin{eqnarray*}
\sum\limits_{i=1}^{n}\frac{n-1}{\min\{\widehat{\beta}_i,1\}}\prod\limits_{j=1}^{i-1}\frac{b_{jj}}{\bar{\beta}_j}<
\frac{n-1}{\min\{\widehat{\beta},1\}}\sum\limits_{i=1}^{n}\prod\limits_{j=1}^{i-1}\frac{b_{jj}}{\bar{\beta}_j}=
(n-1)\alpha<\frac{n-1}{\beta}\leqslant\frac{n-1}{\min\{\beta,1\}}.
\end{eqnarray*}
Similarly, for $\widehat{\beta}<1$ and
$\alpha\beta<\widehat{\beta}$, the conclusion can be proved
directly.
\end{proof}
\section{Numerical examples}
Two examples are given to show that the bound in Theorem
\ref{GL-mian} is sharper than those in Theorems \ref{G-B} and
\ref{L-C}.
\begin{example} \label{ex1} Consider the family of $B$-matrices in \cite{Li1}:
 \[M_k =
\left[ \begin{array}{cccc}
 1.5 &0.5   &0.4 &0.5 \\
 -0.1  &1.7    &0.7 &0.6   \\
 0.8  &-0.1\frac{k}{k+1}    &1.8 &0.7 \\
 0 & 0.7  &0.8  & 1.8
\end{array} \right],\]
where $k\geqslant 1$. Then $M_k=B_k^{+}+C_k$, where
\[B^+ =
\left[ \begin{array}{cccc}
  1 &0 &-0.1 &0 \\
 -0.8   &1    &0 &-0.1   \\
  0   &-0.1\frac{k}{k+1}-0.8    &1 &-0.1\\
 -0.8 & -0.1 &0  &  1
\end{array}\right].\]
By computations, we have $
\beta=\frac{1}{10(k+1)},\bar{\beta}_1=\bar{\beta}_2=\frac{90k+91}{100k+100},\bar{\beta}_3=0.99,\bar{\beta}_4=1,$
$\hat{\beta}_1=\frac{820k+828}{900k+900},\hat{\beta}_2=0.99,\hat{\beta}_3=1$
and $\hat{\beta}_4=1$. Then it is easy to verify that $M_k$
satisfies the condition (ii) of Theorem \ref{th5}. Hence, by Theorem
\ref{G-B} (Theorem 2.2 in \cite{Ga}), we have
\[\max\limits_{d\in
[0,1]^n}||(I-D+DM)^{-1}||_{\infty}\leqslant
\frac{4-1}{\min\{\beta,1\}}=30(k+1).\] It is obvious that
\[30(k+1)\longrightarrow +\infty, when ~k\longrightarrow +\infty.\]
By Theorem \ref{L-C}, we have that for any $k\geqslant 1$,
\begin{eqnarray*}
&&\max\limits_{d\in [0,1]^n}||(I-D+DM)^{-1}||_{\infty}\\&\leqslant&
3\left(\frac{1}{\bar{\beta}_1}+\frac{1}{\bar{\beta}_2}\cdot\frac{1}{\bar{\beta}_1}+\frac{1}{\bar{\beta}_3}\cdot\frac{1}{\bar{\beta}_1\bar{\beta}_2}
+\frac{1}{\bar{\beta}_1\bar{\beta}_2\bar{\beta}_3}\right)\\&=&3\left(\frac{100k+100}{90k+91}+\frac{(100k+100)^2}{(90k+91)^2}
+\frac{2(100k+100)^2}{0.99(90k+91)^2}\right).
\end{eqnarray*}
By Theorem \ref{GL-mian}, we have that for any $k\geqslant 1$,
\begin{eqnarray*}
&&\max\limits_{d\in [0,1]^n}||(I-D+DM)^{-1}||_{\infty}\\&\leqslant&
3\left(\frac{1}{\hat{\beta}_1}+\frac{1}{\hat{\beta}_2}\cdot\frac{1}{\bar{\beta}_1}+\frac{1}{\bar{\beta}_1\bar{\beta}_2}
+\frac{1}{\bar{\beta}_1\bar{\beta}_2\bar{\beta}_3}\right)\\&=&3\left(\frac{900k+900}{820k+828}+\frac{(100k+100)}{0.99(90k+91)}
+\frac{1.99(100k+100)^2}{0.99(90k+91)^2}\right)\\&<&3\left(\frac{100k+100}{90k+91}+\frac{(100k+100)^2}{(90k+91)^2}
+\frac{2(100k+100)^2}{0.99(90k+91)^2}\right).
\end{eqnarray*}
In particular, when $k=1$,
\[3\left(\frac{900k+900}{820k+828}+\frac{(100k+100)}{0.99(90k+91)}
+\frac{1.99(100k+100)^2}{0.99(90k+91)^2}\right)\approx 13.9878,\]
\[3\left(\frac{100k+100}{90k+91}+\frac{(100k+100)^2}{(90k+91)^2}
+\frac{2(100k+100)^2}{0.99(90k+91)^2}\right)\approx 14.3775,\] and
the bound (\ref{G-bound}) in Theorem \ref{G-B} is
\[\frac{4-1}{\min\{\beta,1\}}=30(k+1)=60.\]
When $k=2$,
\[3\left(\frac{900k+900}{820k+828}+\frac{(100k+100)}{0.99(90k+91)}
+\frac{1.99(100k+100)^2}{0.99(90k+91)^2}\right)\approx 14.0265,\]
\[3\left(\frac{100k+100}{90k+91}+\frac{(100k+100)^2}{(90k+91)^2}
+\frac{2(100k+100)^2}{0.99(90k+91)^2}\right)\approx 14.4246,\] and
the bound (\ref{G-bound}) in Theorem \ref{G-B} is
\[\frac{4-1}{\min\{\beta,1\}}=30(k+1)=90.\]
\end{example}
\begin{example} \label{ex2} Consider the following family of $B$-matrices:
 \[M_k =
\left[ \begin{array}{cc}
 \frac{1}{k} &\frac{-a}{k}   \\
 0  &\frac{1}{k}
\end{array} \right],\]
where $\frac{\sqrt{5}-1}{2}<a<1$ and $\frac{2-a^2}{1+a}<k<1$. Then
$M_k=B_k^{+}+C$ with $C$ is the null matrix. By simple computations,
we can get
\[
\beta=\frac{1-a}{k},\bar{\beta}_1=\frac{1-a^2}{k},\bar{\beta}_2=\frac{1}{k},\hat{\beta}_1=\frac{1}{k}
~and~\hat{\beta}_2=\frac{1}{k}.\] It is not difficult to verify that
$M_k$ satisfies the condition (i) of Theorem \ref{th5}. Thus, the
bound (\ref{lc-bound}) of Theorem \ref{L-C} (Theorem 2.4 in
\cite{Li3}) is
\begin{eqnarray*}
\sum\limits_{i=1}^{n}\frac{n-1}{\min\{\bar{\beta}_i,1\}}\prod\limits_{j=1}^{i-1}\frac{b_{jj}}{\bar{\beta}_j}=
\frac{k+1}{1-a^2},
\end{eqnarray*}
which is larger than the bound
\[\frac{1}{\min\{\beta,1\}}=\frac{k}{1-a}\]
given by (\ref{G-bound}) in Theorem \ref{G-B} (Theorem 2.2 in
\cite{Ga}). However, by Theorem \ref{GL-mian} we can get that
\begin{eqnarray*}
\max\limits_{d\in [0,1]^n}||(I-D+DM)^{-1}||_{\infty}&\leqslant&
\frac{2-a^{2}}{1-a^{2}}.
\end{eqnarray*}
which is smaller than the bound (\ref{G-bound}) in Theorem
\ref{G-B}, i.e.,
\[\frac{2-a^{2}}{1-a^{2}}<\frac{k}{1-a}.\]
In particular, when $a=\frac{4}{5}$ and $k=\frac{8}{9}$, the bounds
in Theorems \ref{G-B} and \ref{L-C} are respectively
\[\frac{1}{\min\{\beta,1\}}=\frac{k}{1-a}=\frac{360}{81}\]
 and \[\sum\limits_{i=1}^{n}\frac{n-1}{\min\{\bar{\beta}_i,1\}}\prod\limits_{j=1}^{i-1}\frac{b_{jj}}{\bar{\beta}_j}=\frac{k+1}{1-a^2}=\frac{425}{81},\]
while the bound (\ref{Gl-bound}) in Theorem \ref{GL-mian} is
\[\sum\limits_{i=1}^{n}\frac{n-1}{\min\{\hat{\beta}_i,1\}}\prod\limits_{j=1}^{i-1}\frac{b_{jj}}{\bar{\beta}_j}=\frac{2-a^{2}}{1-a^{2}}=\frac{306}{81}.\]
These two examples show that the bound in Theorem \ref{GL-mian} is
sharper than those in Theorems \ref{G-B} and \ref{L-C}.
\end{example}


\section*{Acknowledgements}
This work is partly supported by National Natural Science
Foundations of China (11601473 and 31600299), Young Talent fund of
University Association for Science and Technology in Shaanxi, China
(20160234), and CAS 'Light of West China' Program.


\end{document}